\documentclass{amsart}
\usepackage{amsmath}
\usepackage{amscd}
\usepackage{amssymb}
\newcommand{\cal}{\mathcal}
\renewcommand{\div}{\operatorname{div}}

\newcommand{\Vol}{\operatorname{Vol}}

\newtheorem{theorem}{Theorem}[section]
\newtheorem{theorem/definition}{Theorem/Definition}[section]
\newtheorem{proposition}{Proposition}[section]
\newtheorem{lemma}{Lemma}[section]
\newtheorem{corollary}{Corollary}[section]
\theoremstyle{remark}
\newtheorem{remark}{Remark}[section]
\theoremstyle{definition}

\begin{document}
\title
{On Bach-Flat Gradient Shrinking Ricci Solitons}
\author{Huai-Dong Cao and Qiang Chen}
\address{Department of Mathematics\\ Lehigh University\\
Bethlehem, PA 18015} \email{huc2@lehigh.edu; qic208@lehigh.edu}

\begin{abstract}
In this paper, we classify $n$-dimensional ($n\ge 4$) complete
Bach-flat gradient shrinking Ricci solitons. More precisely, we
prove that any $4$-dimensional Bach-flat gradient shrinking Ricci
soliton is either Einstein, or locally conformally flat hence a finite quotient of the Gaussian
shrinking soliton $\mathbb R^4$ or the round cylinder $\mathbb
S^3\times \mathbb R$. More generally, for $n\ge 5$, a Bach-flat
gradient shrinking Ricci soliton is either Einstein, or a finite
quotient of the Gaussian shrinking soliton $\mathbb R^n$ or the
product $N^{n-1}\times \mathbb R$, where $N^{n-1}$ is Einstein.

\end{abstract}

\maketitle
\date{}

\footnotetext[0]{

The research of the first author was partially supported by NSF Grant
DMS-0909581.}

\section{The results}

A complete Riemannian manifold $(M^n, g_{ij})$ is called a {\it gradient Ricci soliton} 
if there exists a smooth function $f$ on $M^n$ such that the Ricci tensor $R_{ij}$ of the
metric $g_{ij}$ satisfies the equation
$$R_{ij}+\nabla_i\nabla_jf=\rho g_{ij}$$
for some constant $\rho$. For $\rho=0$ the Ricci soliton is {\it
steady}, for $\rho>0$ it is {\it shrinking} and for $\rho<0$ {\it
expanding}. The function $f$ is called a {\it potential function}
of the gradient Ricci soliton. Clearly, when $f$ is a constant the
gradient Ricci soliton is simply an Einstein manifold. Thus Ricci
solitons are natural extensions of Einstein metrics.  Gradient
Ricci solitons play an important role in Hamilton's Ricci flow as
they correspond to self-similar solutions, and often arise as
singularity models. Therefore it is important to classify gradient
Ricci solitons or understand their geometry.

In this paper we shall focus our attention on  gradient shrinking
Ricci solitons, which are possible Type I singularity models in
the Ricci flow. We normalize the constant $\rho=1/2$ so that the shrinking soliton equation is given by

\begin{equation*}
R_{ij}+\nabla_i\nabla_jf=\frac{1}{2} g_{ij}. \eqno(1.1)
\end{equation*}
In recent years, inspired by Perelman's work \cite{P1, P2}, much efforts have been devoted to study the 
geometry and classifications of gradient shrinking Ricci solitons.  We refer the reader to the survey papers \cite{Cao2010, Cao2011} by the first 
author and the references therein for recent progress on the subject. In particular, it is known (cf. \cite {P2, NW1, CCZ08}) 
that any complete 3-dimensional gradient shrinking Ricci soliton is a finite quotient of  either the round sphere 
$\mathbb S^3$, or the Gaussian shrinking soliton $\mathbb R^3$, or the round cylinder $\mathbb
S^2\times \mathbb R$. For higher dimensions, it has been proven that complete 
locally conformally flat gradient shrinking Ricci solitons are finite quotients of either the round sphere 
$\mathbb S^n$, or the Gaussian shrinking soliton $\mathbb R^n$, or the round cylinder $\mathbb
S^{n-1}\times \mathbb R$ (first due to Z. H. Zhang \cite{Zh2} based on the work of Ni-Wallach \cite{NW1}, see also the works of Eminenti-La Nave-Mantegazza  \cite {ELM}, 
Petersen-Wylie \cite{PW3}, X. Cao, B. Wang and Z. Zhang \cite{CWZ08}, and  Munteanu-Sesum \cite{MS}). 
Moreover, it follows from the works of Fern\'andez-L\'opez and Garc\'ia-R\'io \cite{FG},  and
Munteanu-Sesum \cite{MS} that $n$-dimensional complete
gradient shrinking solitons with harmonic Weyl tensor are rigid in the sense that they are finite quotients of the product of an Einstein manifold $N^k$ with
the Gaussian shrinking soliton $\mathbb R^{n-k}$.   

Our aim in this paper is to investigate an interesting class of
complete gradient shrinking Ricci solitons: those with vanishing
Bach tensor.  This well-known tensor was introduced by R. Bach \cite{Bach} in
early 1920s' to study conformal relativity.  On any
$n$-dimensional manifold $(M^n, g_{ij})$ ($n\ge 4$), the Bach
tensor is defined by
$$B_{ij} =\frac 1 {n-3}\nabla^k\nabla^l W_{ikjl}+\frac 1 {n-2}
R_{kl}W{_i}{^k}{_j}^l.$$
Here $W_{ikjl}$ is the Weyl tensor. It is easy to see that if $(M^n, g_{ij})$
is either locally conformally flat (i.e., $W_{ikjl}=0$) or Einstein, then $(M^n, g_{ij})$ is {\it Bach-flat}: $B_{ij}=0$. 

The case when $n=4$ is the most interesting,  as it is well-known (cf.
\cite{Besse} or \cite{Der}) that on any compact 4-manifold $(M^4,
g_{ij})$, Bach-flat metrics are precisely the critical points of
the following {\sl conformally invariant} functional on the space
of metrics,
$$\cal W (g)=\int_M |W_g|^2 dV_g,$$ where $W_g$ denotes the Weyl tensor of $g$.
Moreover, if $(M^4, g_{ij})$ is either {\it half conformally flat } (i.e., self-dual or
anti-self-dual) or {\it locally conformal to an Einstein
manifold}, then its Bach tensor vanishes.  In this paper, we shall see the (stronger) converse
holds for gradient shrinking solitons: {\it Bach-flat $4$-dimensional
gradient shrinking solitons are either Einstein or locally
conformally flat}.

Our main results are the following classification theorems for Bach-flat
gradient shrinking Ricci solitons:

\begin{theorem} Let $(M^4, g_{ij}, f)$ be a complete Bach-flat gradient  shrinking Ricci soliton. Then,  $(M^4, g_{ij}, f)$ is either

\medskip
(i) Einstein, or

\medskip
(ii) locally conformally flat, hence a finite quotient of either the Gaussian shrinking soliton $\mathbb
R^{4}$ or  the round cylinder $\mathbb S^3\times
\mathbb R$.
\end{theorem}

\medskip
More generally, for $n\ge 5$, we have:

\begin{theorem} Let $(M^n, g_{ij}, f)$ ($n\ge 5$) be a complete Bach-flat gradient  shrinking Ricci soliton. Then,  $(M^n, g_{ij}, f)$ is either

\medskip
(i) Einstein, or

\medskip
(ii) a finite quotient of the Gaussian shrinking soliton $\mathbb
R^{n}$, or

\medskip
(iii) a finite quotient of $N^{n-1}\times \mathbb R$, where
$N^{n-1}$ is an Einstein manifold of positive scalar curvature.
\end{theorem}

\medskip
The basic idea in proving Theorem 1.1 and Theorem 1.2 is to
explore hidden relations between the Bach tensor $B_{ij}$ and the
Cotton tensor $C_{ijk}$ on a gradient shrinking Ricci soliton. It
turns out that the key link between these two classical tensors is provided
by a third tensor, the covariant 3-tensor  $D_{ijk}$ defined by 
$$ D_{ijk} = \frac{1} {n-2} (A_{jk} \nabla_i f- A_{ik} \nabla_jf) + \frac {1} {(n-1)(n-2)} (g_{jk}E_{il} - g_{ik}E_{jl})\nabla_lf, \eqno(1.2)$$
where $A_{ij}$ is the Schouten tensor and $E_{ij}$ is the Einstein tensor (see Section 3). 

This tensor  $D_{ijk}$ (and its equivalent version in Section 3) was introduced by the authors in
\cite{CC09} to study the classification of locally conformally
flat gradient steady solitons. On one hand,  for any gradient
Ricci soliton, it turns out that the Bach tensor $B_{ij}$ can be expressed in terms of $D_{ijk}$ and the 
Cotton tensor $C_{ijk}$: $$B_{ij} =-\frac 1{n-2}(\nabla_k D_{ikj}+\frac{n-3}{n-2}C_{jli}\nabla_l f).\eqno(1.3)$$
On the other hand, as shown in \cite{CC09},  $D_{ijk}$ is closely
related to the Cotton tensor and the Weyl tensor by
$$D_{ijk}=C_{ijk}+W_{ijkl}\nabla_l f. \eqno(1.4)$$
By using (1.3),  we are able to  show that  the vanishing of the
Bach tensor $B_{ij}$ implies the vanishing of $D_{ijk}$ for
gradient shrinking solitons (see Lemma 4.1). On the other hand, the norm of  $D_{ijk}$ is linked to the geometry 
of level surfaces of the potential function
$f$ by the following key identity (see Proposition 3.1): at any point $p\in M^n$ where $\nabla f(p)\ne 0$,
$$|D_{ijk}|^2=\frac{2|\nabla f|^4}{(n-2)^2} |h_{ab}-\frac{H}{n-1}g_{ab}|^2 +\frac{1}{2(n-1)(n-2)}
|\nabla_a R|^2,\eqno(1.5)$$ where $h_{ab}$ and $H$ are the second
fundamental form and the mean curvature for the level surface
$\Sigma=\{f=f(p)\}$, and $g_{ab}$ is the induced metric on the level surface $\Sigma$. Thus, the vanishing of $D_{ijk}$ and (1.5)
tell us that the geometry of the shrinking Ricci soliton and the
level surfaces of the potential function are very special
(see Proposition 3.2), and consequently we deduce that $D_{ijk}=0$
implies  the Cotton tensor $C_{ijk}=0$ at all points where
$|\nabla f|\ne 0$ (see Lemma 4.2 and Theorem 5.1).
Furthermore, when $n=4$, by using (1.4) we can actually show that
the Weyl tensor $W_{ijkl}$ must vanish at all points where $|\nabla f|\ne 0$ (see Lemma 4.3).
Then the main theorems follow immediately from the known classification theorem for
 locally conformally flat gradient shrinking Ricci solitons and the rigid theorem for  
gradient shrinking Ricci solitons with harmonic Weyl tensor respectively.

\begin{remark} Very recently, by cleverly using the tensor $D_{ijk}$, X. Chen and Y. Wang \cite{CW} have shown 
that 4-dimensional half-conformally flat  gradient shrinking Ricci solitons are
either Einstein, or locally conformally flat. Since half-conformal flat implies Bach-flat in dimension 4, our
Theorem 1.1 is clearly an improvement.
 \end{remark}

\medskip
Note that by  a theorem of Hitchin (cf. \cite{Besse}, Theorem
13.30), a compact 4-dimensional half-conformally flat Einstein
manifold (of positive scalar curvature) is $\mathbb S^4$ or
$\mathbb CP^{2}$. Combining Hitchin's theorem and Theorem 1.1, we
arrive at the following classification of 4-dimensional
compact half-conformally flat  gradient shrinking Ricci solitons which was
first obtained by X. Chen and Y. Wang \cite{CW}:

\begin{corollary} If $(M^4, g_{ij}, f)$ is a compact half-conformally flat gradient  shrinking Ricci soliton, then $(M^4, g_{ij})$ is isometric
to the standard $\mathbb S^4$ or $\mathbb CP^2$.
\end{corollary}

Finally, in Section 5, we observe that for all gradient (shrinking, or steady, or expanding) Ricci solitons, the vanishing of $D_{ijk}$ implies the vanishing of the Cotton tensor
$C_{ijk}$ at all points where $|\nabla f|\ne 0$ (see Theorem 5.1). This yields the classification of $n$-dimensional ($n\ge 4$) gradient shrinking Ricci solitons, as well as $4$-dimensional 
gradient steady Ricci solitons,  with vanishing $D_{ijk}$.  

\begin{theorem} Let $(M^n, g_{ij}, f)$ ($n\ge 4$) be a complete gradient {\sl shrinking} Ricci soliton with 
$D_{ijk}=0$,  then 

\medskip
(i)  $(M^4, g_{ij}, f)$ is either Einstein, or a finite quotient of $\mathbb R^4$ or $\mathbb S^3\times
\mathbb R$; 

\medskip
(ii) for $n\ge 5$, $(M^n, g_{ij}, f)$ is either Einstein, or a finite quotient of the Gaussian shrinking soliton $\mathbb
R^{n}$, or  a finite quotient of  $N^{n-1}\times \mathbb R$, where
$N^{n-1}$ is Einstein.
\end{theorem}

\begin{theorem} Let $(M^4, g_{ij}, f)$  be a complete gradient {\sl steady} Ricci soliton with $D_{ijk}=0$,  then 
 $(M^4, g_{ij}, f)$ is either Ricci flat or isometric to the Bryant soliton.
\end{theorem}

\medskip 

\noindent {\bf Acknowledgments}.  We are very grateful to Professor S.-T. Yau for suggesting us to consider 4-dimensional self-dual gradient shrinking Ricci solitons in summer 2010 which  
in part inspired us to study Bach-flat gradient shrinking Ricci solitons. We would also like to thank Professor Richard Hamilton for his interest in our work, and the referee for 
very helpful comments and suggestions which made our paper more readable.

\section{Preliminaries}

In this section, we fix our notations and recall some basic facts
and known results about gradient Ricci solitons that we shall need
in the proof of Theorem 1.1 and Theorem 1.2.

First of all, we recall that on any $n$-dimensional Riemannian
manifold $(M^n, g_{ij} )$ ($n\ge 3 $),  the Weyl curvature tensor
is given by

\begin{align*}
W_{ijkl}  = & R_{ijkl} - \frac{1}{n-2}(g_{ik}R_{jl}-g_{il}R_{jk}-g_{jk}R_{il}+g_{jl}R_{ik})\\
& + \frac{R}{(n-1)(n-2)} (g_{ik}g_{jl}-g_{il}g_{jk}), \\
\end{align*}
and  the Cotton tensor by
 $$C_{ijk}=\nabla_i R_{jk}-\nabla_j R_{ik}-\frac {1}{2(n-1)} ( g_{jk} \nabla_i R - g_{ik} \nabla_j R).$$
 
\begin{remark} In terms of the Schouten tensor 
$$ A_{ij}= R_{ij}- \frac{R}{2(n-1)}g_{ij}, \eqno(2.1)$$  
$$ W_{ijkl}=R_{ijkl} - \frac{1}{n-2} (g_{ik}A_{jl}-g_{il}A_{jk}-g_{jk}A_{il}+g_{jl}A_{ik}), $$ 
$$C_{ijk}=\nabla_i A_{jk}-\nabla_j A_{ik}.$$
\end{remark}
 
 It is well known that, for $n=3$, $W_{ijkl}$ vanishes identically, while $C_{ijk}=0$ if and only if
$(M^3, g_{ij})$ is locally conformally flat; for $n\ge 4$,
$W_{ijkl}=0$ if and only if  $(M^n, g_{ij})$ is locally
conformally flat.  Moreover, for $n\ge 4$, the Cotton tensor
$C_{ijk}$ is, up to a constant factor, the divergence of the Weyl
tensor:
$$C_{ijk}=-\frac{n-2}{n-3} \nabla_l W_{ijkl}, \eqno(2.2)$$
hence the vanishing of the Cotton tensor $C_{ij k} = 0$ (in
dimension $n\ge 4$) is also referred as being harmonic Weyl.

Moreover, for $n\ge 4$, the Bach tensor is defined by
$$B_{ij}=\frac 1 {n-3}\nabla^k\nabla^l W_{ikjl}+\frac 1 {n-2}
R_{kl}W{_i}{^k}{_j}^l. $$ By (2.2), we have
$$B_{ij}=\frac 1 {n-2} (\nabla_k C_{kij} + R_{kl}W{_i}{^k}{_j}^l). \eqno(2.3)$$

Note that  $C_{ijk}$ is skew-symmetric in the
first two indices and trace-free in any two indices:
$$ C_{ijk}=-C_{jik} \quad \mbox{and} \quad g^{ij}C_{ijk}=g^{ik}C_{ijk}=0.\eqno(2.4)$$

Next we recall some  basic facts about complete gradient shrinking
Ricci solitons satisfying Eq. (1.1) .

\begin{lemma} {\bf (Hamilton \cite{Ha95F})} Let $(M^n, g_{ij}, f)$
be a complete gradient shrinking Ricci soliton satisfying Eq.
(1.1). Then we have
$$\nabla_iR=2R_{ij}\nabla_jf, \eqno(2.5)$$ and
$$R+|\nabla f|^2-f=C_0 $$ for some constant $C_0$. Here $R$
denotes the scalar curvature.
\end{lemma}

Note that if we normalize $f$ by adding the constant $C_0$ to it,
then we have
$$R+|\nabla f|^2=f. \eqno(2.6)$$

\begin{lemma} Let $(M^n, g_{ij}, f)$ be a complete gradient steady soliton. Then it has nonnegative
scalar curvature $R\ge 0$.
\end{lemma}

Lemma 2.2 is a special case of a more general result of B.-L. Chen
\cite{BChen} which states that $R\ge 0$ for any ancient solution
to the Ricci flow. For an alternative proof of Lemma 2.2, see,
e.g., the more recent work of Pigola-Rimoldi-Setti \cite{PRS}.

\begin{lemma} {\bf (Cao-Zhou \cite{CZhou08})} Let $(M^n, g_{ij}, f)$
be a complete noncompact gradient shrinking Ricci soliton
satisfying (1.1) and the normalization (2.6). Then,

\medskip
(i) the potential function $f$ satisfies the estimates
$$\frac{1}{4} (r(x)-c_1)^2\leq f(x)\leq \frac{1}{4} (r(x)+c_2)^2,$$
where $r(x)=d(x_0, x)$ is the distance function from some fixed
point $x_0\in M$, $c_1$ and $c_2$ are positive constants depending
only on $n$ and the geometry of $g_{ij}$ on the unit ball $B(x_0,
1)$;

\medskip
(ii) there exists some constant $C>0$ such that
$$ \Vol(B(x_0, s))\leq C s^{n}$$
for $s>0$ sufficiently large.
\end{lemma}

\section{The covariant 3-tensor $D_{ijk}$}

In this section, we review the covariant 3-tensor $D_{ijk}$ introduced in our previous work
\cite{CC09} and its important properties.  

For any gradient Ricci soliton  satisfying the defining equation
$$R_{ij}+\nabla_i\nabla_jf=\rho g_{ij},\eqno (3.1)$$ the covariant 3-tensor $D_{ijk}$ is defined
as:
\begin{align*}
 D_{ijk} = & \frac{1}{n-2}(R_{jk} \nabla_i f- R_{ik} \nabla_j f) +\frac{1}{2(n-1)(n-2)} (g_{jk}\nabla_i R-g_{ik} \nabla_j R)\\
              & - \frac{R}{(n-1)(n-2)} (g_{jk}\nabla_i f  - g_{ik} \nabla_j f ).
 \end{align*}
Note that, by using (2.5), $D_{ijk}$ can also be expressed as 
$$ D_{ijk} = \frac{1} {n-2} (A_{jk} \nabla_i f- A_{ik} \nabla_jf) + \frac {1} {(n-1)(n-2)} (g_{jk}E_{il} - g_{ik}E_{jl})\nabla_lf, \eqno(3.2)$$
where $A_{ij}$ is the Schouten tensor in (2.1) and $E_{ij}=R_{ij}-\frac {R} {2} g_{ij}$ is the Einstein tensor. 

This 3-tensor $D_{ijk}$ is closely tied to the Cotton tensor  and
played a significant role in our previous work \cite{CC09} on classifying
locally conformally flat gradient steady solitons, as well as in
the subsequent works of S. Brendle \cite{Brendle} and X. Chen and
Y. Wang \cite{CW}.

\begin{lemma} 
Let $(M^n, g_{ij}, f)$ ($n\ge 3$) be a complete gradient soliton
satisfying (3.1). Then $D_{ijk}$ is related to the Cotton tensor
$C_{ijk}$ and the Weyl tensor $W_{ijkl}$ by
$$ D_{ijk}=C_{ijk}+W_{ijkl}\nabla_l f.$$
\end{lemma}

\begin{proof} From the soliton equation (3.1), we have
$$\nabla_i R_{jk}-\nabla_j R_{ik}=-\nabla_i\nabla_j \nabla_k f +\nabla_j\nabla_i \nabla_k f=-R_{ijkl}\nabla_l f.$$
Hence,  using (2.5), we obtain
\begin{align*}
 C_{ijk} =&\nabla_i R_{jk}-\nabla_j R_{ik}-\frac {1}{2(n-1)} ( g_{jk} \nabla_i R - g_{ik} \nabla_j R)\\
= & -R_{ijkl} \nabla_l f - \frac{1}{(n-1)} (g_{jk}R_{il}-g_{ik} R_{jl})\nabla_l f\\
             = & -W_{ijkl} \nabla_l f -\frac{1}{n-2}(R_{ik} \nabla_j f- R_{jk} \nabla_i f)\\
                &  + \frac{1}{2(n-1)(n-2)} (g_{jk}\nabla_i R -g_{ik} \nabla_j R)
             +\frac{R}{(n-1)(n-2)} (g_{ik} \nabla_j f - g_{jk}\nabla_i f)\\
             = & -W_{ijkl} \nabla_l f +D_{ijk}.
 \end{align*}
\end{proof}

\begin{remark}
By Lemma 3.1, $D_{ijk}$ is equal to the Cotton tensor $C_{ijk}$ in dimension $n=3$.  In addition, it is easy to see that 
$$D_{ijk}\nabla_k f=C_{ijk} \nabla_k f
$$
Also, $D_{ijk}$ vanishes if $(M^n, g_{ij}, f)$ ($n\ge 3$) is either Einstein or locally conformally flat. 
Moreover, like the Cotton tensor $C_{ijk}$, $D_{ijk}$ is skew-symmetric in the first two indices and trace-free in any 
two indices:   
$$ D_{ijk}=-D_{jik} \quad \mbox{and} \quad g^{ij}D_{ijk}=g^{ik}D_{ijk}=0. \eqno(3.3)$$ 
\end{remark}

What is so special about $D_{ijk}$ is the following key identity,
which links  the norm of $D_{ijk}$ to the geometry of the level
surfaces of the potential function $f$. We refer readers to \cite{CC09} (cf. Lemma 4.4 in \cite{CC09}) for its proof.

\begin{proposition} {\bf (Cao-Chen \cite{CC09})}  Let $(M^n, g_{ij}, f)$ ($n\ge 3$) be an $n$-dimensional
gradient Ricci soliton satisfying (3.1). Then, at any point $p\in
M^n$ where $\nabla f(p)\ne 0$, we have
$$|D_{ijk}|^2=\frac{2|\nabla f|^4}{(n-2)^2} |h_{ab}-\frac{H}{n-1}g_{ab}|^2 +\frac{1}{2(n-1)(n-2)}
|\nabla_a R|^2,$$ where $h_{ab}$ and $H$ are the second
fundamental form and the mean curvature of the level surface
$\Sigma=\{f=f(p)\}$, and $g_{ab}$ is the induced metric on $\Sigma$.
\end{proposition}

Finally, thanks to Proposition 3.1, the vanishing of $D_{ijk}$ implies
many nice properties about the geometry of the Ricci soliton
$(M^n, g_{ij}, f)$ and the level surfaces of the potential
function $f$.

\begin{proposition}  
Let $(M^n, g_{ij}, f)$ ($n\ge 3 $) be any complete gradient Ricci
soliton with  $D_{ijk}=0$, and let  $c$ be a regular value of $f$
and $\Sigma_c=\{f=c\}$ be the level surface of $f$. Set
$e_1=\nabla f /|\nabla f |$ and pick any orthonormal frame $e_2,
\cdots, e_n$ tangent to the level surface $\Sigma_c$. Then

 \medskip

(a) $|\nabla f|^2$ and the scalar curvature $R$ of $(M^n,
g_{ij}, f)$ are constant on $\Sigma_c$;

\smallskip
(b) $R_{1a}=0$ for any $a\ge 2$ and $e_1=\nabla f /|\nabla f |$
is an eigenvector of $Rc$;

\smallskip

(c) the second fundamental form $h_{ab}$ of $\Sigma_c$ is of the
form $h_{ab}=\frac{H}{n-1} g_{ab}$;

\smallskip

(d) the mean curvature $H$ is constant on $\Sigma_c$;

\smallskip

(e) on $\Sigma_c$, the Ricci tensor of $(M^n, g_{ij}, f)$ either
has a unique eigenvalue $\lambda$, or has two distinct eigenvalues
$\lambda$ and $\mu$ of multiplicity $1$ and $n-1$ respectively. In
either case, $e_1=\nabla f /|\nabla f |$ is an eigenvector of
$\lambda$

\end{proposition}

\begin{proof} Clearly (a) and (c) follow immediately from $D_{ijk}=0$, Proposition 3.1, and (2.6);

(b) follows from (a) and (2.5): $R_{1a}=\frac{1} {2|\nabla
f|}\nabla_a R=0$;

For (d), we consider the Codazzi equation
$$R_{1cab}=\nabla_a^{\Sigma_c}h_{bc}-\nabla_b^{\Sigma_c}h_{ac}, \qquad a, b, c=2, \cdots, n. \eqno(3.4)$$ 
Tracing over b and c in (3.4),  we obtain
$$R_{1a}=\nabla_a^{\Sigma_c} H- \nabla_b^{\Sigma_c}h_{ab}=(1-\frac{1}{n-1})\nabla_a H.$$
Then (d) follows since $R_{1a}=0$;

Finally, the second fundamental form is given by
$$h_{ab}=<\nabla_a \frac{\nabla f}{|\nabla f|},e_b>=\frac{\nabla_a\nabla_b f}{|\nabla f|}=\frac{\rho g_{ab}-R_{ab}}{|\nabla f|}.$$
Combining this with (c),  we see that $$R_{ab} =\rho g_{ab}-|\nabla
f|h_{ab}=(\rho -\frac{H}{n-1}|\nabla f|)  g_{ab}.$$ But both $H$
and $|\nabla f|$ are constant on $\Sigma_c$, so the Ricci tensor
restricted to the tangent space of $\Sigma_c$ has only one
eigenvalue $\mu$: $$\mu=R_{aa}=\rho - H |\nabla f|/(n-1),  \quad a=2, \cdots, n, \eqno(3.5)$$ which is constant along
$\Sigma_c$. On the other hand,
$$\lambda=R_{11}=R-\sum_{a=2}^n R_{aa}=R- (n-1)\rho + H|\nabla f|, \eqno(3.6)$$ again
a constant along $\Sigma_c$. This proves (e).

\end{proof}

\begin{remark} 
In any neighborhood $U$ of the level surface $\Sigma_c$ where  $|\nabla f|^2\neq 0$,  we can always express the metric $g_{ij}$
as $$ds^2=\frac {1}{|\nabla f|^2(f, \theta)}(df)^2  + g_{ab}(f, \theta) d\theta^{a} d\theta^{b}.
\eqno(3.7)$$
Here $\theta=(\theta^{2}, \cdots, \theta^{n})$ denotes any local coordinates on $\Sigma_c$. It follows from Proposition 3.2 that, when $D_{ijk}=0$, 
the metric $g_{ij}$ is in fact a warped product metric on U of the form: $$ ds^2=dr^2+\varphi^2(r)\bar g_{\Sigma_{c}}, \eqno(3.8)$$ where   
$\bar g_{\Sigma_{c}}$ denotes the induced metric on $\Sigma_c$. Furthermore, $(\Sigma_c, \bar g_{\Sigma_{c}})$ is necessarily Einstein. The 
details can be found in \cite{Cao et al}. 

\end{remark}

\section{The proof of Theorem 1.1 and Theorem 1.2}

Throughout this section,  we assume that $(M^n, g_{ij}, f)$ ($n\ge
4$) is a complete gradient shrinking soliton satisfying (1.1).

First of all, we relate the Bach tensor $B_{ij}$ to the Cotton
tensor $C_{ijk}$ and the tensor $D_{ijk}$, and then show that the Bach-flatness implies $D_{ijk}=0$.

\begin{lemma} Let $(M^n, g_{ij}, f)$ be a complete gradient shrinking soliton. If $B_{ij}=0$, then $D_{ijk}=0$.
\end{lemma}

\begin{proof} 

By direct computations, and using (2.2), (2.3) and Lemma 3.1, we have

\begin{align*}
B_{ij} &=-\frac 1 {n-2} \nabla_k C_{ikj}+\frac 1 {n-2} R_{kl}W_{ikjl}\\
&=-\frac 1{n-2}\nabla_k(D_{ikj}-W_{ikjl}\nabla_l f)+\frac 1 {n-2} R_{kl}W_{ikjl}\\
&=-\frac 1{n-2}(\nabla_k D_{ikj}-\nabla_k W_{jlik}\nabla_l f) +\frac 1 {n-2} (R_{kl}+\nabla_k\nabla_lf)W_{ijkl}.\\
\end{align*}
Hence, 
$$B_{ij} =-\frac 1{n-2}(\nabla_k D_{ikj}+\frac{n-3}{n-2}C_{jli}\nabla_l f). \eqno(4.1)$$

Next, we use (4.1) to show that Bach flatness implies the vanishing of the tensor $D_{ijk}$.  By Lemma 2.3, for each $r>0$ sufficiently large, 
$\Omega_r=\{x\in M | f(x) \le r\}$ is compact.  Now by the definition of $D_{ijk}$, the identity (4.1), as well as properties (2.4) and (3.3), we have

\begin{align*}
\int_{\Omega_r} B_{ij}\nabla_i f\nabla_j f dV&= -\frac {1} {(n-2)}\int_{\Omega_r} \nabla_k D_{ikj} \nabla_i f\nabla_j f dV\\
&=\frac {1} {(n-2)}\big(\int_{\Omega_{r}} D_{ikj} \nabla_i f \nabla_k\nabla_j f dV - \int_{\Omega_{r}} \nabla_k (D_{ikj}\nabla_i f\nabla_jf)dV\big)\\
&=-\frac {1} {(n-2)} \big(\int_{\Omega_{r}} D_{ikj} \nabla_i f R_{jk} dV + \int_{\partial \Omega_r} D_{ikj}\nabla_if \nabla_jf \nu_k dS\big)\\
&=-\frac {1} {2(n-2)}\int_{\Omega_{r}} D_{ikj} (\nabla_i f R_{jk} -\nabla_k f R_{ij}) dV \\
&=-\frac {1} {2}\int _{\Omega_{r}} |D_{ikj}|^2 dV.
\end{align*}
Here we have used the fact, in view of (3.3), that 
$$ \int_{\partial \Omega_r} D_{ikj}\nabla_if \nabla_jf \nu_k dS=\int_{\partial \Omega_r} D_{ikj}\nabla_if\nabla_jf \nabla_kf\frac{1}{|\nabla f|}dS=0.$$
By taking $r\to \infty$, we immediately obtain  
$$\int_M B_{ij}\nabla_i f\nabla_j f dV =-\frac {1} {2}\int_M |D_{ikj}|^2 dV.$$
This completes the proof of Lemma 4.1.

\end{proof}

\begin{lemma} Let $(M^n, g_{ij}, f)$ ($n\ge 4$) be a complete gradient shrinking Ricci soliton with vanishing $D_{ijk}$, then the Cotton tensor $C_{ijk}=0$ at all points
where $\nabla f\ne 0$.
\end{lemma}

\begin{proof}

First of all, $D_{ijk}=0$ and Lemma 3.1 imply
$$C_{ijk}=-W_{ijkl}\nabla_l f, \eqno(4.2)$$ hence

$$C_{ijk}\nabla_k f=-W_{ijkl}\nabla_k f\nabla_l f=0. \eqno(4.3)$$

Next, for any point $p\in M$ with $\nabla f(p)\ne 0$, we
choose a local coordinates system $(\theta^2,
\cdots, \theta^n)$ on the lever surface $\Sigma=\{f=f(p)\}$. In any neighborhood $U$ of the level surface $\Sigma$ where  $|\nabla f|^2\neq 0$, we
use the local coordinates system 
$$(x^1, x^2, \cdots, x^n)=(f, \theta^2, \cdots,
\theta^n)$$ adapted to level surfaces. In the
following, we use $a,b,c$ to represent indices on the level sets
which ranges from 2 to n, while $i,j,k$ from 1 to n. Under the above
chosen local coordinates system,
the metric $g$ can be expressed as
$$ds^2=\frac{1}{|\nabla f|^2} df^2+g_{ab}(f, \theta) d\theta^a
d\theta^{b}.$$

Next, we denote by $\nu=-\frac{\nabla f}{|\nabla f|}$.  It is
then easy to see that
$$\nu=-|\nabla f|\partial_f, \qquad \mbox{or} \qquad \partial_f=\frac {1} {|\nabla f|^2}\nabla f. $$
Also $\partial_1$ and $\partial_f$ shall be interchangeable below. And we have $$\nabla_1 f=1,\qquad  \mbox{and} \qquad  \nabla_a f=0 \ \mbox{for} \ a\ge 2.
\eqno(4.4)$$ Then, in this coordinate, (4.3) implies that
$$C_{ij1}=0.$$ 

{\bf Claim 1}: $D_{ijk}=0$ implies $C_{abc}=0$ for $a\ge 2, b\ge 2$, and $c\ge 2.$\\

 To show $C_{abc}=0$, we make use
of Proposition 3.2 as follows: from the Codazzi equation (3.4) and $h_{ab}=Hg_{ab}/(n-1)$, we get

$$R_{1cab}=\nabla_a^{\Sigma}h_{bc}-\nabla_b^{\Sigma}h_{ac}=\frac
{1} {n-1} (g_{bc} \partial_a(H) -g_{ac}\partial_b(H)).\eqno(4.5)$$

But we also know that the mean curvature $H$ is constant on the
level surface $\Sigma$ of $f$, so
$$R_{1abc}=0.$$ Moreover, since $R_{1a}=0$, we easily obtain $$W_{1abc}=R_{1abc}=0. $$ By (4.2), we
have  $$C_{abc}=-W_{abci}\nabla_j f g^{ij}=W_{1cab}\nabla_1 f g^{11}=0.$$ This
finishes the proof of Claim 1. \\

{\bf Claim 2}: $D_{ijk}=0$ implies $C_{1ab}=C_{a1b}=0$. \\

To do so, let us  compute the second fundamental form in the
preferred local coordinates system $(f, \theta^2, \cdots,
\theta^n)$:
\begin{align*}
h_{ab}=-<\nu,\nabla_a\partial_b>=-<\nu,\Gamma^1_{ab}\partial_f>=\frac
{\Gamma^1_{ab}}{|\nabla f|}. 
\end{align*}
But the Christoffel symbol $\Gamma^1_{ab}$ is given by
\begin{align*}
\Gamma^1_{ab}=\frac{1}{2}g^{11}(-\frac{\partial g_{ab}}{\partial
f})=\frac 1 2 |\nabla f| \nu(g_{ab}).
\end{align*}
Hence, we obtain

$$h_{ab}=\frac 1 2 \nu(g_{ab}). \eqno(4.6)$$

On the other hand, since $|\nabla f|$ is constant along level
surfaces, we have
$$[\partial_a,\nu]=-[\partial_a,| \nabla f|\partial_f]=0.$$
Then using the fact that $<\nu,\nu>=1$ and $<\nu,\partial_a>=0$,
it is easy to  see that
$$\nabla_\nu\nu=0. \eqno(4.7)$$

By direct computations and using Proposition 3.2,  we can compute the
following component of the Riemannian curvature tensor:
\begin{align*}
Rm(\nu,\partial_a,\nu,\partial_b)
& = <\nabla_{\nu}\nabla_a\partial_b-\nabla_a\nabla_{\nu}\partial_b,\nu>\\
&=<\nabla_{\nu}({\nabla^{\Sigma}}_a\partial_b+\nabla^{\perp}_a\partial_b),\nu>-<\nabla_a\nabla_{\nu}\partial_b,\nu>\\
&=<{\nabla^{\Sigma}}_a\partial_b,-\nabla_{\nu}\nu>+<\nabla_{\nu}(-h_{ab}\nu),\nu>+<\nabla_b\nu,\nabla_a\nu>\\
&=-\nu(h_{ab})+h_{ac} h_{cb}\\
&=-\frac {\nu(H)}{n-1} g_{ab} +\frac {H^2} {({n-1)}^2} g_{ab}.\\
\end{align*}
Taking trace in $a, b$ yields
$$ Rc(\nu,\nu) = -\nu(H)+\frac{H^2}{n-1}.$$
Thus
\begin{align*}
Rm(\nu,\partial_a,\nu,\partial_b)&=-\frac {\nu(H)}{n-1} g_{ab} +\frac {H^2} {(n-1)^2} g_{ab}\\
&=\frac{Rc(\nu,\nu)}{n-1}g_{ab}.\\
\end{align*}

Finally, we are ready to compute $C_{1ab}$:
$$C_{1ab}=-W_{1abi}\nabla_j f g^{ij}=W_{1a1b}|\nabla
f|^2=W(\nu,\partial_a,\nu,\partial_b). \eqno(4.8)$$ 
However, by using proposition 3.2(e), we have:
\begin{align*}
W(\nu,\partial_a,\nu,\partial_b)
&=Rm(\nu,\partial_a,\nu,\partial_b)+\frac{R g_{ab}}{(n-1)(n-2)}
-\frac 1 {n-2}(Rc(\nu,\nu)g_{ab}+R_{ab})\\
&=\frac{Rc(\nu,\nu)}{n-1}g_{ab}+\frac{R g_{ab}}{(n-1)(n-2)}-\frac 1 {n-2}(Rc(\nu,\nu)g_{ab}+ R_{ab})\\
&=\frac{\lambda}{n-1}g_{ab}+\frac{(\lambda+(n-1)\mu)g_{ab}}{(n-1)(n-2)}-\frac 1{n-2} (\lambda g_{ab}+\mu g_{ab})\\
&=0.\\
\end{align*}
Hence, $$C_{1ab}=W_{1a1b}=0. \eqno(4.9)$$ 
This finishes the proof of Claim 2.

Therefore we have shown that  $C_{ij1}=0$, $C_{abc}=0$ and $C_{1ab}=0$. This proves Lemma 4.2. 

\end{proof}

For dimension $n=4$, we can prove a stronger result:
\begin{lemma} Let $(M^4, g_{ij}, f)$ be a complete gradient shrinking Ricci soliton with vanishing $D_{ijk}$, then the Weyl tensor $W_{ijkl}=0$ at all points
where $\nabla f\ne 0$.
\end{lemma}
\begin{proof}
From Lemma 4.2 we know that $D_{ijk}=0$ implies $C_{ijk}=0$. Hence it follows from
Lemma 3.1 that
$$W_{ijkl}\nabla_l f=0$$ for all $1\le i, j, k, l\le 4$. For any
$p$ where $|\nabla f|\ne 0$, we can attach an orthonormal frame at $p$ with $e_1=\frac {\nabla f}{|\nabla f|}$, and then we have
$$W_{1ijk}(p)=0, \qquad \mbox{for} \ \ 1\le i, j, k \le 4. \eqno(4.10)$$ Thus it remains to show
$$W_{abcd}(p)=0$$ for all $2\le a, b, c,d \le 4$. However, this
essentially reduces to showing the Weyl tensor is zero in 3
dimensions (cf. \cite{Ha82}, p.276--277): observing that the Weyl
tensor $W_{ijkl}$ has all the symmetry of the $R_{ijkl}$  and is
trace free in any two indices. Thus, 
$$W_{2121}+W_{2222}+W_{2323}+W_{2424}=0,$$ and so, by (4.10),
$$W_{2323}=-W_{2424}.$$
Similarly, we have
$$W_{2424}=-W_{3434}=W_{2323},$$ which implies $W_{2323}=0.$ On the
other hand,
$$W_{1314}+ W_{2324}+W_{3334}+W_{4344}=0,$$ so $W_{2324}=0.$ This
shows that $W_{abcd}=0$ unless $a, b, c, d$ are all distinct.
But, there are only three choices for the indices $a, b, c, d$ as they range from 2 to 4.
\end{proof}

\medskip

Now we are ready to finish the proof of our main theorems: \\

\noindent {\bf Conclusion of the proof of Theorem 1.1}:  \ Let
$(M^4, g_{ij}, f)$ be a complete Bach-flat gradient shrinking
Ricci soliton. Then, by Lemma 4.1, we know $D_{ijk}=0$. We divide the arguments into two cases: \\

$\bullet$ Case 1:  the set $\Omega =\{p\in M | \nabla f(p)\ne 0\}$
is dense. 

\smallskip
By Lemma 4.1 and Lemma 4.3, we know that $W_{ijkl}=0$ on
$\Omega$. By continuity, we know that $W_{ijkl}=0$ on $M^4$. Therefore we
conclude that $(M^4, g_{ij}, f)$ is locally conformally flat.
Furthermore, according to the classification result for locally
conformally flat gradient shrinking Ricci solitons mentioned in
the introduction, $(M^4, g_{ij}, f)$ is a finite quotient of
either $\mathbb R^4$, or $\mathbb S^3\times
\mathbb R$. \\

$\bullet$  Case 2:  $|\nabla f|^2=0$ on some nonempty open set. In
this case, since any gradient shrinking Ricci soliton is analytic
in harmonic coordinates, it follows  that $|\nabla f|^2=0$ on $M$,
i.e., $(M^4, g_{ij})$ is Einstein.

This completes the proof of Theorem 1.1. 

\qed

\medskip
\noindent {\bf Conclusion of the proof of Theorem 1.2}: Let $(M^n,
g_{ij}, f)$, $n\ge 5$, be a Bach-flat gradient shrinking Ricci soliton. Then, by Lemma 4.1, Lemma 4.2
and the same argument as in the proof of Theorem 1.1 above, we know that $(M^n, g_{ij}, f)$ either is Einstein,
or has harmonic Weyl tensor. In the latter case, by the rigidity
theorem of Fern\'andez-L\'opez and Garc\'ia-R\'io \cite{FG}  and
Munteanu-Sesum \cite{MS} for harmonic Weyl tensor, $(M^n, g_{ij},
f)$ is either Einstein or isometric to a finite quotient of of $N^{n-k}\times \mathbb
R^{k}$ ($k>0$) the product of an Einstein manifold $N^{n-k}$ with the
Gaussian shrinking soliton $\mathbb R^{k}$. However, Proposition
3.2 (e) says that the Ricci tensor either has one unique
eigenvalue or two distinct eigenvalues with multiplicity of 1 and
$n-1$ respectively. Therefore, only $k=1$ and $k=n$ can occur in
$N^{n-k}\times \mathbb R^{k}$.

\qed

\medskip

\section{Gradient Ricci solitons with vanishing $D_{ijk}$}

First of all, we notice that the proofs of Lemma 4.2 and Lemma 4.3 are valid for gradient steady and expanding Ricci solitons as well. Hence we have the following
general result.

\begin{theorem} Let $(M^n, g_{ij}, f)$ ($n\ge 4$) be a complete non-trivial gradient Ricci soliton
satisfying (3.1) and with $D_{ijk}=0$,  then 

\medskip
(i) the Weyl tensor $W_{ijkl}=0$ for $n=4$, i.e.,  $(M^n, g_{ij}, f)$ is locally conformally flat;

\medskip
(ii) the Cotton tensor $C_{ijk}=0$ for $n\ge
5$, i.e.,  $(M^n, g_{ij}, f)$ has  harmonic Weyl tensor.
\end{theorem}

As an immediate consequence of Theorem 5.1, the classification theorem for locally conformally flat gradient shrinking solitons and  
the  rigidity theorem for gradient shrinking solitons with harmonic Weyl tensor mentioned in the introduction, and Proposition 3.2 (e),
we have the following rigidity theorem for gradient shrinking Ricci solitons with vanishing $D_{ijk}$: 

\begin{corollary} Let $(M^n, g_{ij}, f)$ ($n\ge 4$) be a complete gradient {\sl shrinking} Ricci soliton with 
$D_{ijk}=0$,  then 

\medskip
(i)  $(M^4, g_{ij}, f)$ is either Einstein, or a finite quotient of $\mathbb R^4$ or $\mathbb S^3\times
\mathbb R$; 

\medskip
(ii) for $n\ge 5$, $(M^n, g_{ij}, f)$ is either Einstein, or a finite quotient of the Gaussian shrinking soliton $\mathbb
R^{n}$, or  a finite quotient of  $N^{n-1}\times \mathbb R$, where
$N^{n-1}$ is Einstein.
\end{corollary} 

Moreover, combining Theorem 5.1 (i) and the 4-d classification theorem for locally conformally flat 
gradient steady Ricci solitons \cite{CC09, CM}, we have 

\begin{corollary} Let $(M^4, g_{ij}, f)$  be a complete gradient {\sl steady} Ricci soliton with $D_{ijk}=0$,  then 
 $(M^4, g_{ij}, f)$ is either Ricci flat or isometric to the Bryant soliton.
\end{corollary} 

\medskip

Finally, let us further examine the relations among $D_{ijk}$, $C_{ijk},$, $W_{ijkl}$ and $B_{ij}$. Note that Theorem 5.1(ii) tells 
us that for any nontrivial gradient Ricci soliton, $D_{ijk}=0$ implies $C_{ijk}=0$. On the other hand, the converse is 
not true because the product space $\mathbb{S}^{k}\times {\mathbb R}^{n-k}$ has $C_{ijk}=0$ but not $D_{ijk}=0$ by Proposition 3.2(e) for $k\ge 2$ and $n-k\ge 2$. 
So one naturally would wonder how much stronger is the condition $D_{ijk}=0$ than $C_{ijk}=0$? It turns out that we have several equivalent characterizations of $D_{ijk}=0$. 

\begin{theorem} Let $(M^n, g_{ij}, f)$ ($n\ge 5$) be a nontrivial gradient Ricci soliton satisfying (3.1). Then the following statements are equivalent: 

\smallskip
(a) $D_{ijk}=0$;

\smallskip
(b) $C_{ijk}=0$, and $W_{1ijk}=0$ for $1\le i,j,k\le n$; 

\smallskip
(c) $\div B\cdot \nabla f=0$ and $W_{1a1b}=0$  for $2\le a, b\le n$.
\end{theorem} 

\begin{proof}  $(a) \rightarrow (b)$:  This follows from Theorem 5.1 and Lemma 3.1.

\medskip
$(b) \rightarrow (c)$:  Clearly, it suffices to show that $C_{ijk}=0$ implies $\div B\cdot \nabla f=0$. In fact, $C_{ijk}=0$ implies $\div B=0$ for $n\ge 5$. 
This follows from the following formula, which is well-known at least for $n=4$ among experts in conformal geometry and general relativity:

\begin{lemma}\label{divbach} For $n\ge 4$, we have 
$$
\div B \equiv \nabla_j B_{ij}=\frac{n-4}{(n-2)^2} C_{ijk}R_{jk}\,.
$$
\end{lemma}

\noindent {\sl Proof of Lemma 5.1.}  Recall that we have 
$$
C_{ijk}=\nabla_i A_{jk}-\nabla_j A_{ik}\,,
$$
and 
$$W_{ijkl}=R_{ijkl} - \frac{1}{n-2} (g_{ik}A_{jl}-g_{il}A_{jk}-g_{jk}A_{il}+g_{jl}A_{ik})\,. \eqno(5.1)$$
By using the expression of the Bach tensor in (2.3), we have
$$
(n-2)\nabla_i B_{ij}= \nabla_i \nabla_k (\nabla_kA_{ij}-\nabla_iA_{kj}) + \nabla_k R_{kl} W_{ikjl}+ R_{kl}\nabla_kW_{ikjl}\,.
$$
But, 
\begin{align*}
\nabla_i \nabla_k (\nabla_kA_{ij}-\nabla_iA_{kj})= & (\nabla_i \nabla_k-\nabla_k \nabla_i)\nabla_kA_{ij}\\
 =& -R_{il}\nabla_lA_{ij} +R_{kl}\nabla_kA_{lj}+R_{ikjl}\nabla_kA_{il}\\
= &R_{ikjl}\nabla_kA_{il}\,.
\end{align*}
Thus, by using (5.1),  
\begin{align*}
\nabla_i \nabla_k (\nabla_kA_{ij}-\nabla_iA_{kj}) + \nabla_k R_{kl} W_{ikjl}= & (R_{ikjl}-W_{ikjl})\nabla_kA_{il}\\
=& \frac{1}{n-2} (A_{jk}g_{il}C_{lki} +A_{ik}C_{kji})\\
=&-\frac{1}{n-2} R_{ki}C_{jki}\,.
\end{align*}
Moreover, by (2.2), we know 
$$
\nabla_kW_{ikjl}=\frac{n-3}{n-2}C_{jlk}\,.
$$
Summing up, we obtain 
$$
(n-2)\nabla_i B_{ij}=\frac{n-4}{n-2}R_{kl}C_{jkl}\,.
$$
\qed

$(c) \rightarrow (a)$: by Lemma 5.1,  Lemma 3.1 and (3.3),  we have
\begin{align*}
\div B\cdot \nabla f & =\frac{n-4}{(n-2)^2} C_{ijk}R_{jk}\nabla_i f \\ 
& =\frac{n-4}{(n-2)^2} (D_{ijk}-W_{ijkl}\nabla_l f) R_{jk}\nabla_i f \\
& =\frac{n-4}{2(n-2)} |D_{ijk}|^2+\frac{n-4}{(n-2)^2} W_{1a1b} R_{ab}|\nabla f|^2.  
\end{align*}
Thus, $\div B\cdot \nabla f =0$ and  $W_{1a1a}=0$ for $2\le a\le n$ imply $D_{ijk}=0$ for all $1\le i,j,k\le n$. 

\smallskip
This completes the proof of Theorem 5.2. 

\end{proof}


\begin{thebibliography}{99}

\bibitem{Bach} Bach, R. \textit {Zur Weylschen Relativit\:atstheorie und der Weylschen Erweiterung des Kr\:ummungstensorbegriffs},  Math.Z. \textbf {9} (1921),
 110--135.

\bibitem{Besse} Besse, A., Einstein Manifolds. {\sl Springer-Verlag, Berlin}, 1987. MR0867684 

\bibitem{Brendle} Brendle, S., \textit{Uniqueness of gradient Ricci solitons}, Math. Res. Lett. \textbf {18} (2011), no. 3, 531--538. MR2802586 

\bibitem{Cao2010} Cao, H.-D., {\em Recent progress on Ricci
solitons}, Recent advances in geometric analysis, 1--38, Adv. Lect. Math. (ALM), {\bf 11} Int. Press, Somerville, MA, 2010.

\bibitem{Cao2011} Cao, H.-D., \textit{Geometry of complete gradient shrinking Ricci solitons}, Geometric and Analysis (Vol I), 227-246, Adv. Lect. Math. (ALM), {\bf 17}, 
Int. Press, Somerville, MA, 2011.

\bibitem{Cao et al} Cao, H.-D., Catino, G.,  Chen, Q., Mantegazza, C., and  Mazzieri, L., {\em Bach-flat gradient steady Ricci solitons}, arXiv:1107.4591

\bibitem{CCZ08} Cao, H.-D., Chen, B.-L. and Zhu, X.-P., {\em Recent
developments on Hamilton's Ricci flow},  Surveys in differential
geometry, Vol. XII,  47--112, Surv. Differ. Geom., \textbf{XII},
Int. Press, Somerville, MA, 2008.

\bibitem{CC09} Cao, H.-D. and Chen, Q., {\em On locally conformally flat gradient steady solitons}, 
Trans. Amer. Math. Soc. \textbf {364} (2012), no. 5, 2377--2391. MR2888210


\bibitem{CZhou08} Cao, H.-D. and Zhou, D., {\it On complete
gradient shrinking solitons},  J. Differential Geom. \textbf{85}
(2010), 175--185. MR2732975 

\bibitem{CWZ08} Cao, X., Wang, B. and Zhang, Z., {\sl On locally conformally flat
gradient shrinking Ricci solitons}, Commun. Contemp. Math. \textbf{13} (2011), no. 2, 269--282.

\bibitem{CM} Catino, G. and Mantegazza, C., {\sl Evolution of the Weyl tensor under the Ricci flow},  Ann. Inst. Fourier. \textbf{61} (2011),  no. 4, 1407--1435. 

\bibitem{BChen} Chen, B.-L., {\sl Strong uniqueness of the Ricci
flow},  J. Differential Geom. {\bf 82} (2009), 363-382. MR2520796


\bibitem{CW} Chen, X.X. and Wang, Y., \textit{On four-dimensional anti-self-dual gradient Ricci solitons}, arXiv:1102.0358, 2011 (2011)

\bibitem{Der} Derdzinski, A., {\it Self-dual K\:ahler manifolds and Einstein manifolds of dimension
four},  Compositio Math. \textbf {49} (1983), 405--433.

\bibitem{ELM} Eminenti, M., La Nave, G. and Mantegazza. C., {\sl
Ricci solitons: the equation point of view},  Manuscripta Math.
\textbf{127}  (2008),  345--367.

\bibitem{FG} Fern\'andez-L\'opez, M. and Garc\'ia-R\'io, E., {\it Rigidity of shrinking Ricci solitons}, Math. Z. \textbf{269} (2011), no. 1-2, 461--466. MR2836079 

\bibitem{Ha82}  Hamilton, R. S., {\sl Three-manifolds with positive
Ricci curvature}, J. Diff. Geom. {\bf 17} (1982), 255--306.

\bibitem{Ha95F}  Hamilton, R. S., {\sl The formation of singularities in
the Ricci flow}, Surveys in Differential Geometry (Cambridge, MA,
1993), {\bf 2}, 7-136, International Press, Cambridge, MA, 1995.
MR1375255 

 \bibitem{Kot} Kotschwar, B., {\sl On rotationally invariant shrinking Ricci solitons}, Pacific
 J. Math. {\bf 236} (2008), 73--88. MR2398988 

\bibitem{MS} Munteanu, O. and Sesum, N., {\sl On gradient Ricci solitons}, to appear in J. Geom. Anal.

\bibitem{NW1} Ni, L. and Wallach, N., {\sl On a classification of the gradient shrinking
solitons},  Math. Res. Lett.  \textbf{15}  (2008),  941--955. MR2443993 

\bibitem{P1} Perelman, G., {\sl The entropy formula for the Ricci flow and its geometric
applications}, arXiv:math.DG/0211159.

\bibitem{P2} Perelman, G., {\sl Ricci flow with surgery on three
manifolds} arXiv:math.DG/0303109 (2003).

\bibitem{PW3} Petersen, P. and Wylie, P., {\sl On the classification of gradient Ricci
solitons},  Geom. Topol.  \textbf{14}  (2010),  2277--2300. MR2740647 

\bibitem{PRS} Pigola, S., Rimoldi, M. and Setti, A. G., {\sl Remarks on non-compact
gradient Ricci solitons}, Math. Z.  \textbf{268} (2011), no. 3-4, 777--790. MR2818729 

\bibitem{Zh2} Zhang, Z.-H., {\sl Gradient shrinking solitons
with vanishing Weyl tensor}, Pacific J. Math. {\bf  242} (2009),
189--200. MR2525510 


\end{thebibliography}
\end{document}